\documentclass[11pt,a4paper]{article}

\usepackage[utf8]{inputenc}

\usepackage{geometry}
\usepackage{graphicx}

\usepackage{booktabs}
\usepackage{array}
\usepackage{paralist}
\usepackage{verbatim}
\usepackage{subfig}
\usepackage{tabularx}
\usepackage{fancyhdr}
\newcommand{\abs}[1]{\lvert #1 \rvert}
\newcommand{\norm}[1]{\lVert #1 \rVert}

\usepackage[nottoc,notlof,notlot]{tocbibind}
\usepackage[titles,subfigure]{tocloft}

\usepackage{amsmath,amsfonts,amsthm,mathrsfs,amssymb,cite}
\usepackage{color}
\definecolor{rot}{rgb}{0.000,0.000,0.00}

\newcommand{\R}{{\mathbb R}}

\newcommand{\C}{{\mathbb C}}
\newcommand{\s}{{\mathbb S}}
\newcommand{\be}{\begin{eqnarray}}
\newcommand{\ben}{\begin{eqnarray*}}
\newcommand{\en}{\end{eqnarray}}
\newcommand{\enn}{\end{eqnarray*}}
\newcommand{\real}{{\rm Re\,}}
\newcommand{\ima}{{\rm Im\,}}
\newcommand{\G}{\Gamma}

\newtheorem{thm}{Theorem}[section]

\newtheorem{lem}{Lemma}[section]

\theoremstyle{definition}

\theoremstyle{remark}

\numberwithin{equation}{section}

\title{\bf Shape identification in inverse medium scattering problems with a single far-field pattern}

\author{Guanghui Hu\thanks{Weierstrass Institute, Mohrenstr. 39,
10117 Berlin,
Germany. Email: {\tt hu@wias-berlin.de}}\,,
Mikko Salo\thanks{University of Jyvaskyla, Department of Mathematics and Statistics, P.O.\ Box 35 (MaD), FI-40014 University of Jyvaskyla, Finland. Email: \tt{mikko.j.salo@jyu.fi}}\,,
Esa V\!. Vesalainen\thanks{University of Jyvaskyla, Department of Mathematics and Statistics, P.O.\ Box 35 (MaD), FI-40014 University of Jyvaskyla, Finland. Email: \tt{esavesalainen@gmail.com}}
}

\begin{document}
\maketitle

\begin{abstract}
Consider time-harmonic acoustic scattering from a bounded penetrable obstacle $D\subset \R^N$ embedded in a homogeneous background medium. The index of refraction characterizing the
material inside $D$ is supposed to be H\"older continuous near the corners. If $D\subset \R^2$ is a convex polygon, we prove that
its shape and location can be uniquely determined by the far-field pattern incited by a single incident wave at a fixed frequency. In dimensions $N \geq 3$, the uniqueness applies to penetrable scatterers of rectangular type
 with additional assumptions on the smoothness of the contrast. 
 Our arguments are motivated by recent studies on the absence of non-scattering wavenumbers in domains with corners. As a byproduct, we show that the smoothness conditions in previous corner scattering results are only required near the corners.
\end{abstract}

\section{Introduction and main results}\label{Sec:1}

Assume a time-harmonic incident wave is incident onto a bounded penetrable
obstacle $D\subset \R^N$  ($N \geq 2$) embedded in a homogeneous medium.
The incident field $u^{\mathrm{in}}$ may be any non-trivial solution in $L^2_{\mathrm{loc}}(\R^N)$ of  the Helmholtz equation
\ben
\Delta u^{\mathrm{in}}+k^2 u^{\mathrm{in}}=0 \quad\mbox{in}\quad \R^N,
\enn
where $k>0$ is the wavenumber. For instance, the incident wave is allowed to be a plane wave $\exp(ikx\cdot d)$ with incident direction $d\in \s^{N-1}:=\{x\in \R^{N}: |x|=1\}$,  or a Herglotz wave of the form
\ben
u^{\mathrm{in}}(x)=\int_{\s^{N-1}} \exp(ikx\cdot d)\,g(d)\,ds(d),\quad g\in L^2(\s^{N-1}).
\enn In this paper we suppose the scatterer $D$ to be a convex polygon in $\R^2$ or a convex polyhedron in $\R^N$.
The physical properties of the inhomogeneous medium $D$ can be characterized by the refractive index function (or potential) $q(x)$. Without loss of generality we suppose $q(x)= 1$ for $x\in D^e=\R^N\backslash\overline{D}$ due to the homogeneity of the background medium.

Denote by $u=u^{\mathrm{in}}+u^{\mathrm{sc}}$ the total field generated by $u^{\mathrm{in}}$, where
$u^{\mathrm{sc}}$ is the outgoing scattered field which satisfies the Helmholtz equation
$(\Delta+k^2) u^{\mathrm{sc}}=0$ in $D^e$ and
the  Sommerfeld radiation condition
\begin{equation}\label{eq:radiation}
\lim_{|x|\rightarrow \infty} |x|^{\frac{N-1}{2}}\left\{ \frac{\partial u^{\mathrm{sc}}}{\partial |x|}-ik u^{\mathrm{sc}} \right\}=0,
\end{equation}
uniformly in all directions.
The propagation of the total wave is governed by the Helmholtz equation
\begin{equation}\label{eq:Helm}
\Delta u(x)+k^2 q(x) u(x)=0\quad\mbox{in}\quad \R^N.
\end{equation}
Across the interface $\partial D$, we assume the continuity of the total field and its normal derivative (already implicitly contained in the formulation (\ref{eq:Helm})), i.e.,
\be\label{TE}
u^+=u^-,\quad \partial_\nu u^+=\partial_\nu u^-\quad\mbox{on}\;\partial D.
\en
Here the superscripts $(\cdot)^\pm$ stand for the limits taken from outside and inside, respectively, and $\nu\in \s^{N-1}$ is the unit normal on $\partial D$ pointing into $D^e$.
The unique solvability of the scattering problem \eqref{eq:radiation}--\eqref{TE} in $H^2_{\mathrm{loc}}(\R^N)$ is well-known if $q\in L^\infty(\R^N)$ (see e.g.\ \cite[Chapter 8]{CK} or \cite[Chapter 6]{Kirsch_book}). In particular, the Sommerfeld radiation condition (\ref{eq:radiation}) leads to  the asymptotic expansion
\begin{equation}\label{eq:farfield}
u^{\mathrm{sc}}(x)=\frac{e^{ik |x|}}{|x|^{(N-1)/2}}\; u^\infty(\hat x)+\mathcal{O}\left(\frac{1}{|x|^{N/2}}\right),\quad |x|\rightarrow+\infty,
\end{equation}
 uniformly in all directions $\hat x:=x/|x|$, $x\in\mathbb{R}^N$. The function $u^\infty(\hat x)$ is a real-analytic function defined on $\s^{N-1}$ and is referred to as the \emph{far-field pattern} or the \emph{scattering amplitude} for $u^{\mathrm{in}}$.  The vector $\hat{x}\in\s^{N-1}$ is the observation direction of the far field.

This paper concerns the uniqueness in recovering the boundary $\partial D$ (or equivalently, the convex hull of the support of the contrast $q-1$) from the far-field pattern generated by one incident wave at a fixed frequency. {\color{rot} The study on global uniqueness with a single incident plane or point source wave is usually difficult and challenging. For sound-soft or sound-hard obstacles,
such uniqueness results have been obtained within the class of polyhedral or polygonal scatterers;
see e.g., \cite{Rondi05,CY,EY06,EY08,HL14, LHZ06}. The proofs rely heavily on the reflection principles for the Helmholtz equation
with respect to a Dirichlet or Neumann hyperplane and on properties of the incident wave (for instance,
plane waves do not decay at infinity and point source waves are singular). However, the approach of using reflection principles does not apply to penetrable scatterers due to the lack of "reflectible" (Dirichlet or Neumann) boundary conditions for the Helmholtz equation. To the best of our knowledge, uniqueness with one incident
wave is still unknown within the class of non-convex polyhedral obstacles of impedance type.}

Earlier uniqueness results on shape identification in inverse medium scattering were derived by sending plane waves with infinitely many directions at a fixed frequency (see e.g.,\cite{ElHu, Isakov08, Isakov90, KG, Kirsch93}), which results in overdetermined inverse problems.
Uniqueness with a single far-field pattern has been verified in two cases: $D$ is a ball (not necessarily centered at the origin) and $q\equiv q_0 \neq 1$ is a constant in $D$ \cite{HLL}, or $D$ is a convex polygon or polyhedron and $q$ is real-analytic on $\overline{D}$ satisfying $|q-1|>0$ on $\partial D$ \cite{ElHu2015}.
The unique determination of a variable index of refraction $q$ in $\R^N$ from  knowledge of the far-field patterns of all incident plane waves at fixed frequency, or by measuring the Dirichlet-to-Neumann map of the Helmholtz equation, has also been intensively studied. We refer to  \cite{SG, HN, Nac} and the survey \cite{Uhlmann_BMS} for results for $N \geq 3$ and to recent results \cite{B,OYa} for $N=2$.

The purpose of this article is to remove the real-analyticity assumption made in \cite{ElHu2015} on the refractive index. To do this, we employ a different method that is motivated by the recent studies \cite{BLS, PSV} on the absence of non-scattering wavenumbers in corner domains. This method relies on the construction of suitable complex geometrical optics (CGO) solutions to the Helmholtz equation. Recall that $k$ is called a non-scattering wavenumber if there is a nontrivial incident wave whose far-field pattern vanishes identically. If $k$ is a non-scattering wavenumber,
 the functions $w=u^{\mathrm{in}}|_D$ and $u|_D$ solve the interior transmission eigenvalue problem
\be\label{ITP}\left\{\begin{array}{lll}
\Delta w + k^2 w=0,\quad \Delta u+k^2 q u=0&&\mbox{in}\quad D,\\
w=u,\qquad \qquad \ \,\;\partial_\nu w=\partial_\nu u &&\mbox{on}\quad \partial D.
\end{array}\right.
\en
Thus each non-scattering wavenumber is an interior transmission eigenvalue (see the survey \cite{CH}). On the other hand, if $k^2$ is an interior transmission eigenvalue and if the non-trivial solution
 $w$ of (\ref{ITP})  has a real-analytic extension from $D$ to $\mathbb{R}^N$, then $k^2$ is also a non-scattering wavenumber. This implies that, when $k^2$ is a non-scattering wavenumber,
  the Cauchy data of the total field $u$ on $\partial D$ coincide with the Cauchy data of
a real-analytic function which satisfies the Helmholtz equation in a neighborhood of $D$. A similar phenomenon can be observed around a corner point, if two distinct convex polygons or polyhedra generate the same far-field pattern. Therefore, the argument for proving the absence of non-scattering wavenumber can be used for justifying uniqueness in determining the shape of a penetrable scatterer.

Let $D_j$ for $j=1,2$ be two penetrable scatterers with contrasts $q_j$. Denote by $u_j^\infty$ the far-field pattern of the scattered field caused by a fixed incoming wave $u^{\mathrm{in}}$ incident onto $D_j$ with fixed wavenumber $k > 0$. The first uniqueness result is in two dimensions, and applies to convex polygons.

\begin{thm}\label{shape-identification-in-2D}
Let $D_j \subset \mathbb R^2$ for $j=1,2$ be bounded convex polygons. Assume that $q_j\in L^\infty(\mathbb R^2)$ are contrasts such that $q_j\equiv1$ in $D_j^e$, and each vertex of $D_j$ has some neighborhood $U_j$ such that $q_j|_{\overline{D}_j \cap U_j}$ is $C^\alpha$ for some $\alpha>0$. Furthermore, assume that $q_j(O)\neq1$ for each vertex $O$ of $D_j$.
 Then the relation $u_1^\infty=u_2^\infty$ on $\mathbb S^1$ implies that $D_1=D_2$.
\end{thm}

The next result applies in dimensions $N \geq 3$ but requires that the scatterers are closed rectangular boxes, i.e.\ sets of the form $\left[0,a_1\right]\times \cdots \times\left[0,a_N\right]$ for some $a_j > 0$ up to rotations and translations. We write $H^{s,p}$ for the fractional $L^p$ Sobolev space with smoothness index $s$.

\begin{thm}\label{shape-identification-in-3D}
Let $D_j \subset \mathbb R^N$ for $j=1,2$ be two rectangular boxes. Assume that $q_j\in L^\infty(\mathbb R^N)$ are contrasts such that $q_j\equiv1$ in $D_j^e$, and each corner of $D_j$ has some neighborhood $U_j$ such that $q_j|_{\overline{D}_j \cap U_j}$ has regularity $X$ as specified below. Furthermore, assume that $q_j(O)\neq1$ for each corner $O$ of $D_j$. Then the relation $u_1^\infty=u_2^\infty$ on $\mathbb S^{N-1}$ implies that $D_1=D_2$, provided that one of the following assumptions holds.
\begin{enumerate}
\item[(a)] $N = 3$ and $X = C^{\alpha}$ for some $\alpha>1/4$.
\item[(b)] $N \geq 3$ and $X = H^{s,p}$ for some $s, p$ with $1 < p \leq 2$ and $s > N/p$.
\end{enumerate}
\end{thm}

Theorems \ref{shape-identification-in-2D} and \ref{shape-identification-in-3D} are valid for H\"older or Sobolev potentials and avoid the real-analyticity assumption required in \cite{ElHu2015}. The results in dimensions $N \geq 3$ are confined to penetrable scatterers of rectangular type.
It is still open how to prove
Theorem \ref{shape-identification-in-3D} for general convex  polyhedra with H\"older continuous contrasts.
We remark that the above results remain valid for a large class of incident waves which do not vanish identically in a neighborhood of the scatterer. For instance, $u^{\mathrm{in}}$ is also allowed to be
a spherical point source emitted from some source position located in $D^e$.

Our technique improves the regularity conditions of the corner scattering results of \cite{BLS, PSV}. Namely, regularity is only required in a small neighborhood of the corner point, and otherwise the contrasts are only required to be $L^\infty$. We state the results on the absence of non-scattering wavenumbers as follows.
Throughout the paper we write $B_r:=\{x\in\R^N: |x| < r\}$ for $r>0$.

\begin{thm}\label{corner-scattering-in-2D}
Let $q\in L^\infty(\mathbb R^2)$, and let $W\subset\mathbb R^2$ be a closed sector with angle $< \pi$ and with vertex at $O$. Suppose that $q\equiv1$ in $W^e$, that $q-1$ is compactly supported, and that $q|_{W \cap B_r}$ is $C^{\alpha}$ for some $\alpha > 0$ and $r > 0$. Finally, assume that $q(O)\neq1$. Then, with $q$ as the contrast, for any incoming wave $u^{\mathrm{in}}\not\equiv0$, the far-field pattern $u^\infty$ can not vanish identically.
\end{thm}

\begin{thm}\label{rectangular-corner-scattering-in-3D}
Let $q\in L^\infty(\mathbb R^N)$. Suppose that $q\equiv1$ in $W^e$, that $q-1$ is compactly supported, and that $q|_{W \cap B_r}$ has regularity $X$ for some $r>0$, where one of the following conditions holds:
\begin{enumerate}
\item[(a)]
$N=3$, $W=\left[0,\infty\right[^3$, and $X = C^{\alpha}$ for some $\alpha>1/4$.
\item[(b)]
$N \geq 3$, $W=\left[0,\infty\right[^N$, and $X = H^{s,p}$ for some $s, p$ with $1 < p \leq 2$ and $s > N/p$.
\end{enumerate}
Finally, assume that $q(O)\neq1$. Then, with $q$ as the contrast, for any incoming wave $u^{\mathrm{in}}\not\equiv0$, the far-field pattern $u^\infty$ can not vanish identically.
\end{thm}

The subsequent Section \ref{Proof-2D} is devoted to the proofs of the two-dimensional results, i.e., Theorems \ref{shape-identification-in-2D} and \ref{corner-scattering-in-2D}. The unique determination of a rectangular box in any dimension $N \geq 3$, Theorem \ref{shape-identification-in-3D}, will be proved in Section \ref{Proof-3D}. The result of Theorem \ref{rectangular-corner-scattering-in-3D} on non-scattering wavenumbers can be derived by using the same argument as in Theorem \ref{corner-scattering-in-2D} and we omit its proof.

\section{Proofs in two dimensions}\label{Proof-2D}

Denote by $(r,\varphi)$ the polar coordinates in $\R^2$, and by  $B_R$ the disk centered at the origin $O$ with radius $R>0$. For $\varphi_0\in(0,\pi/2)$,  define $W\subset \R^2$ as the infinite sector between the half-lines $\Gamma^\pm:=\{(r,\varphi):\varphi=\pm\varphi_0\}$. The closure of $W$ will be denoted by $\overline{W}$, which is a closed cone in $\R^2$. Set (see Figure \ref{fig-1})
\ben
S_R=W\cap B_R,\quad \Gamma^\pm_R=\Gamma^\pm\cap B_R,\quad \overline{S}_R=\overline{W}\cap B_R,\quad S_R^e=B_R\backslash\overline{S}_R.
\enn

\begin{figure}[htbp]
\centering
\begin{center}
\scalebox{0.25}{\includegraphics{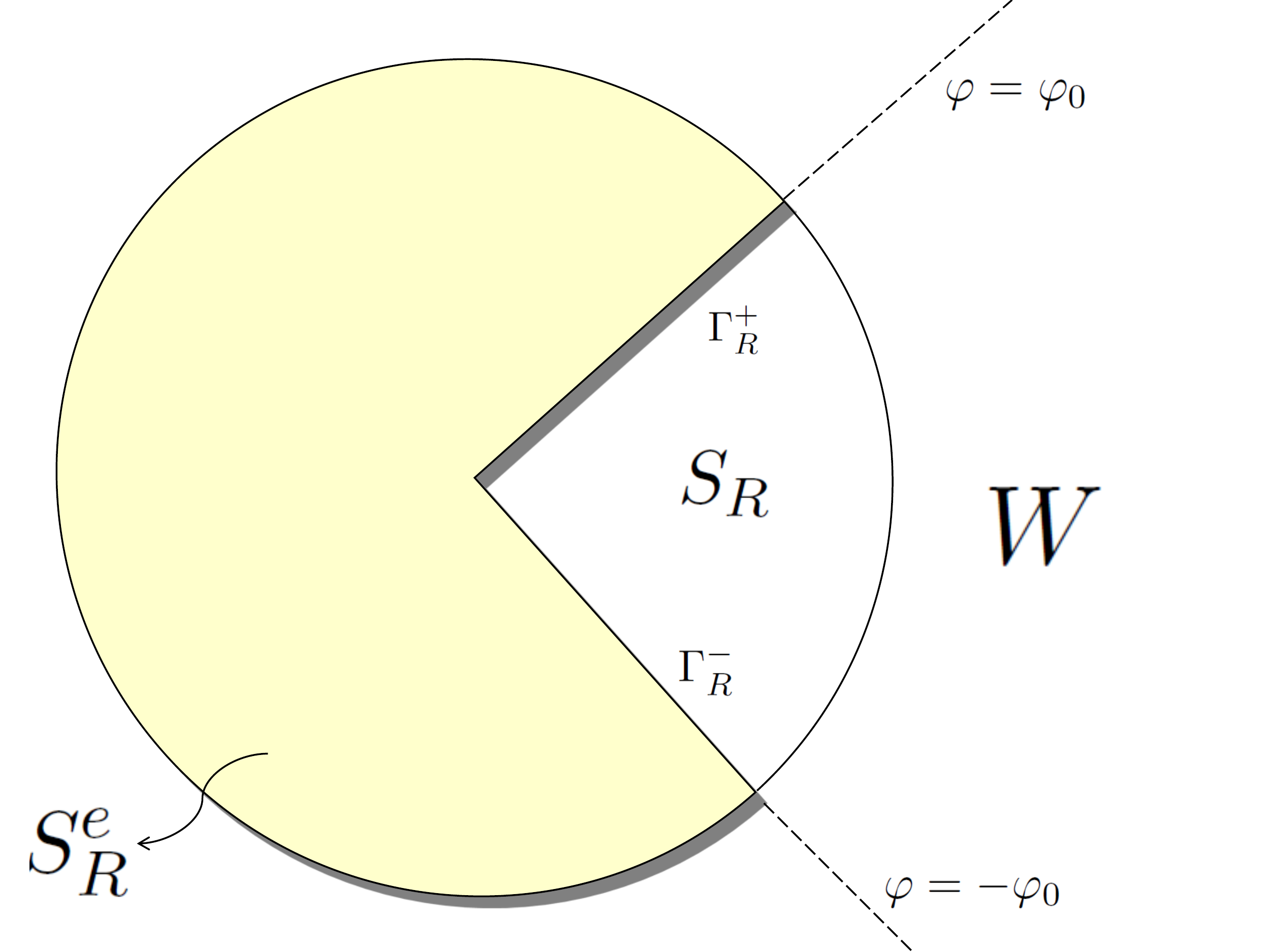}}
 \caption{ Geometrical settings.}
\label{fig-1}
\end{center}
\end{figure}

The following two lemmas are the essential ingredients in the proofs. The first one concerns the construction of suitable  Complex Geometrical Optics (CGO) solutions to the Schr\"odinger equation in $\R^2$. For convenience we employ the common notation $\left\langle x\right\rangle:=(1+|x|^2)^{1/2}$ throughout the paper.
\begin{lem}\label{lem1}
Let $\tilde{q} \in L^{\infty}(\R^2)$ satisfy $\tilde{q}\equiv 1$ in $\R^2\backslash\overline{W}$ and $\left\langle\cdot\right\rangle^\beta\left(\tilde{q}-1\right)\in C^{\alpha}(\overline{W})$ for some $\alpha>0$ and $\beta>5/3$.
If $\rho\in \C^2$ satisfies $\rho\cdot \rho=-k^2$ and $|\ima (\rho)|$ is sufficiently large, then there exists a solution of the Helmholtz equation
\begin{equation}\label{Equation}
\Delta u(x)+k^2 \tilde{q}(x) u(x)=0\quad \mbox{in}\quad \R^2
\end{equation} of the form
\be\label{CGO}
u=e^{-\rho\cdot x}(1+\psi(x)),
\en where $\psi$ satisfies
\be\label{psi}
\|\psi\|_{L^6(\R^2)}=\mathcal{O}(|\ima(\rho)|^{-1/3-\delta})\quad\mbox{as}\quad|\rho|\rightarrow \infty,
\en
for some $\delta>0$.
\end{lem}

Lemma \ref{lem1} is the special case $N=2$ of Lemma 3.1 in \cite{PSV}, the proof of which was based on the uniform Sobolev estimates of Kenig, Ruiz and Sogge \cite{KRS}.
Relying on the construction of CGO solutions of the form  (\ref{CGO}), we next verify a result
for the transmission problem between the Schr\"odinger
equations with constant and piecewise H\"older continuous potentials in a finite polygonal cone.

\begin{lem}\label{lem2}
Suppose $q \in L^{\infty}(B_R)$ satisfies $q|_{\overline{S}_R} \in C^{\alpha}(\overline{S}_R)$ with some $\alpha>0$, and $q\equiv 1$ in $S_R^e$.
Let $v_1, v_2\in H^2(B_R)$ be solutions to
\ben
\Delta v_1(x)+ k^2 v_1(x)=0,\qquad \Delta v_2(x)+k^2q(x) v_2(x)=0\quad \mbox{in}\quad B_R
\enn subject to the transmission conditions
\be\label{Trans}
v_1=v_2,\quad \partial_\nu v_1=\partial_\nu v_2 \quad\mbox{on}\quad \G_R^\pm\, .
\en Then we have  $v_1=v_2\equiv0$ in $B_R$ if $q(O)\neq 1$.
 \end{lem}
 From Lemma \ref{lem2}
it follows that the Cauchy data of non-trivial solutions
of  the Schr\"odinger
equations with  constant and  piecewise H\"older continuous potentials
  cannot coincide on two intersecting lines, if the
potentials involved  do not coincide on the intersection.
The same result was verified in \cite{ElHu2015} but restricted to real-analytic potentials.
Making use of classical corner regularity results for the Laplace equation in the plane (see e.g., \cite[Chapter 1.2]{MNP}, \cite[Chapter 2]{Grisvard} or \cite[Example 16.12]{Dauge}), the approach of Taylor expansion \cite{ElHu2015} can be generalized only to infinitely smooth potentials on
$\overline{S}_R$. Hence, the above Lemma \ref{lem2} has significantly relaxed the regularity assumption used in \cite{ElHu2015}.
Below we carry out the proof of Lemma \ref{lem2} which is valid only when the corner of $S_R$ is convex, i.e., $\varphi_0<\pi/2$.

\begin{proof}[Proof of Lemma \ref{lem2}.] We shall follow the approach from \cite[Section 4]{PSV} but modified to be applicable to a polygonal convex cone with finite height.
  For clarity we divide the proof into three steps.

{\bf Step 1. Establish an orthogonality identity with an exponentially decaying remainder term.}
Set $w=v_1-v_2$. Then $w \in H^2(B_R)$, and we have
\be\label{eq:1}
\Delta w+k^2q w=k^2(q-1)v_1\quad\mbox{in}\quad B_R,\quad w=\partial_{\nu}w=0\qquad \mbox{on}\quad\Gamma^\pm_R\;.
\en
Extending $q$ from $B_{R/2}$ to $\R^2$ in a suitable way,  we get a new potential $\tilde{q} \in L^{\infty}(\R^2)$ satisfying $\tilde{q}|_{\overline{W}} \in C^{\alpha}(\overline{W})$ such that
\ben
\tilde{q}=q\quad\mbox{in}\quad \overline{S}_{R/2}\,,\qquad \tilde{q}\equiv 1\quad\mbox{in}\quad (\overline{W}\backslash \overline{S}_{R})\cup (\R^2\backslash \overline{W}).
\enn
Clearly $\tilde{q}$ fulfills the assumptions in Lemma \ref{lem1}. Set $\beta:=\pi/2-\varphi_0>0$. For any $\varphi$ with $\varphi\in ]-\beta/2,\beta/2\,[$, let
$\omega=(\cos\varphi,\sin\varphi)\in \mathbb S^1$ and let $\omega^\perp_\pm$ be the two vectors orthogonal to $\omega$, i.e.,
$\omega^\perp_\pm=\pm(-\sin\varphi,\cos\varphi)$.
For $\tau>0$, introduce the parameter-dependent vectors $\rho_{\tau,\varphi,\pm}\in \C^2$ as follows
\ben
\rho_{\tau,\varphi,\pm}=\tau \omega+i\,(\tau^2+k^2)^{1/2} \omega^\perp_\pm.
\enn Obviously, $\rho_{\tau,\varphi,\pm}\cdot \rho_{\tau,\varphi,\pm}=-k^2$ and $|\rho_{\tau,\varphi,\pm}|\sim\sqrt{2}\tau$ as $\tau\rightarrow\infty$. By Lemma \ref{lem1},  we may construct solutions to the Schr\"odinger equation (\ref{Equation}) of the form
\be\label{eq:7}
u(x)=u_{\tau,\varphi,\pm}(x)=\exp(-\rho_{\tau,\varphi,\pm}\cdot x)\left(1+\psi_{\tau,\varphi,\pm}(x)\right)\quad\mbox{in}\quad \R^2,
\en
provided $\tau>0$ is sufficiently large.
Applying Green's formula and using (\ref{eq:1}) yields
\be\nonumber
0&=&\int_{S_{R/2}} (\Delta u+k^2\tilde{q}u)w\,dx\\ \nonumber
&=&\int_{S_{R/2}} (\Delta u+k^2q u)w\,dx\\ \nonumber
&=&\int_{S_{R/2}} (\Delta w+k^2 q w)u\,dx+\int_{\partial (S_{R/2})}\left(\partial_{\nu} u\, w-\partial_{\nu}w\,u\right) ds\\ \label{eq:2}
&=&k^2\int_{S_{R/2}} (q-1)v_1\,u\,dx+\int_{\Lambda_{R/2}}\left(\partial_{\nu} u\, w-\partial_{\nu}w\,u\right) ds
\en with $\Lambda_{R/2}:=\{|x|={R/2}\}\cap W$.
Since the constructed CGO solutions decay in $\overline{W}\backslash\{O\}$,
 we shall prove that the integral over $\Lambda_{R/2}$ in (\ref{eq:2}) converges to zero exponentially fast as $\tau\rightarrow\infty$.
\begin{figure}[htbp]
\centering
\begin{center}
\scalebox{0.25}{\includegraphics{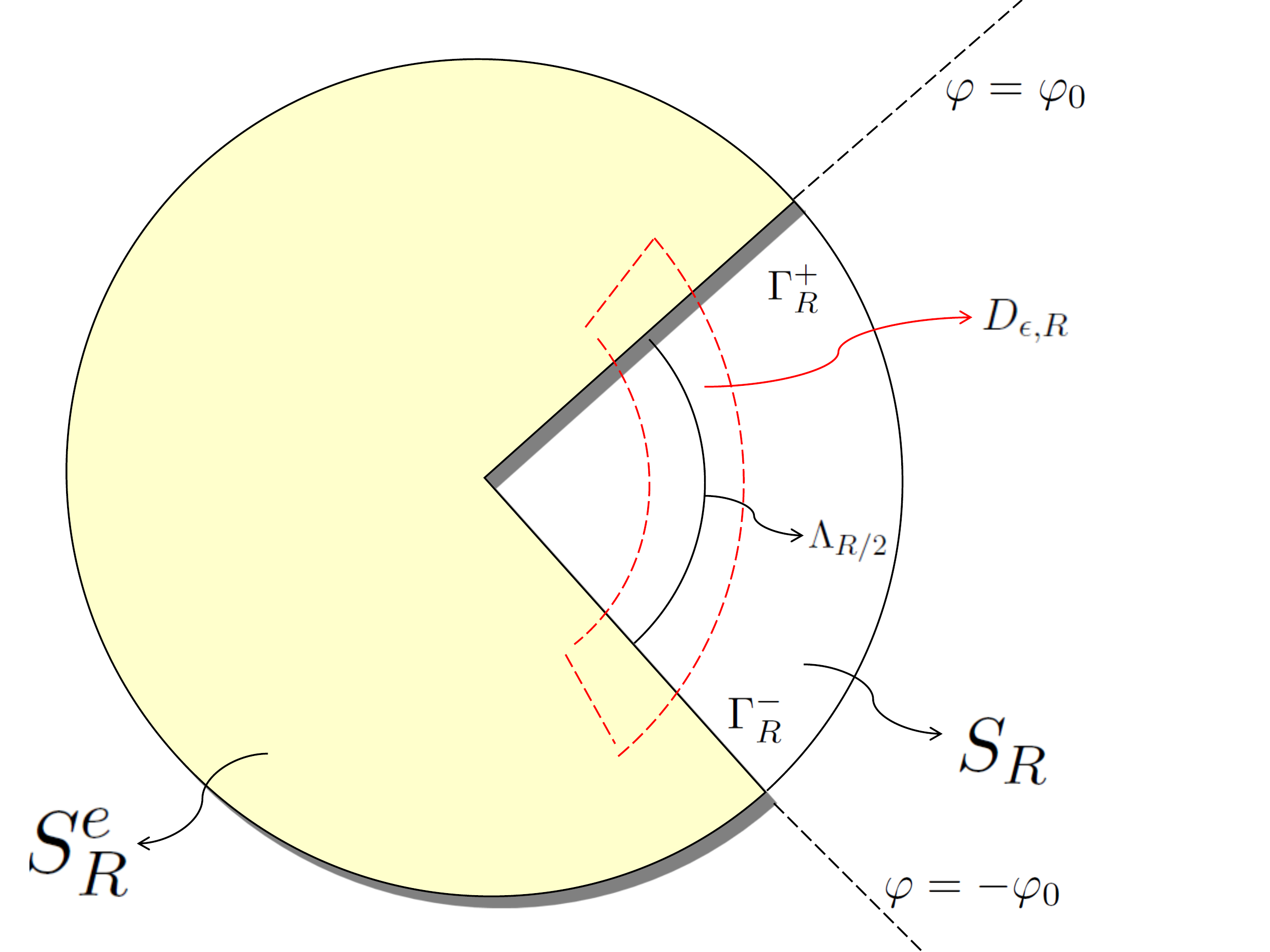}}
 \caption{ Configurations of $\Lambda_{R/2}$ and $D_{\epsilon,R}$ in the proof of Lemma \ref{lem2}.}
\label{fig-2}
\end{center}
\end{figure}

 For $0<\epsilon<\min\{\beta/2,R/2\}$, define a neighborhood of $\Lambda_{R/2}$ by (see Figure \ref{fig-2})
 $$D_{\epsilon, R}:=\{(r,\varphi): R/2-\epsilon<r<R/2+\epsilon, |\varphi|<\varphi_0+\epsilon\}.$$
 Then, there exists $\delta_0=\delta_0(\epsilon,R)>0$ such that
\ben
\real (\rho_{\tau,\varphi,\pm}\cdot x)=\tau(\omega\cdot x)\geq \tau\,\delta_0>0\quad\mbox{for all}\quad x\in D_{\epsilon, R},\quad\varphi\in]-\beta/2,\beta/2\,[.
\enn
This together with the estimates of
$\|\psi_{\tau,\varphi,\pm}\|_{L^6(\R^2)}$ (see (\ref{psi})) implies the exponential decay of the $L^2$-norm of  $u_{\tau,\varphi,\pm}$ over
$L^2(D_{\epsilon, R})$, i.e.,
\ben
\|u_{\tau,\varphi,\pm}\|_{L^2(D_{\epsilon, R})}=\mathcal{O}(e^{-\tau\delta_0})\quad \mbox{as}\quad \tau\rightarrow\infty.
\enn
On the other hand, since $u_{\tau,\varphi,\pm}$ solves the Schr\"odinger equation in $\R^2$, the standard elliptic interior regularity estimate allows us to estimate for $\epsilon'<\epsilon$ that
\ben
\|u_{\tau,\varphi,\pm}\|_{H^2(D_{\epsilon', R})}\leq C\,
\|u_{\tau,\varphi,\pm}\|_{L^2(D_{\epsilon, R})}\leq C e^{-\tau\delta_0}.
\enn
Applying the Cauchy-Schwarz inequality and using the trace lemma, we may estimate the last term on the right-hand side of (\ref{eq:2}) as follows
\ben
\int_{\Lambda_{R/2}}\left((\partial_{\nu} u)\, w-(\partial_{\nu}w)\,u\right) ds&\leq& C\, \|u\|_{H^{3/2}(\Lambda_{R/2})}\,\|w\|_{H^{3/2}(\Lambda_{R/2})}\\
&\leq& C\,\|u\|_{H^2(D_{\epsilon', R})}\left(  \|v_1\|_{H^2(B_{R})}+ \|v_2\|_{H^2(B_{R})}  \right).
\enn
Combining (\ref{eq:2}) with the previous two inequalities,  we get the following orthogonality identity over $S_{R/2}$ with an exponentially decaying remainder term
\be\label{eq:3}
\int_{S_{R/2}} (q-1)v_1\,u_{\tau,\varphi,\pm}\,dx = \mathcal{O}(e^{-\tau\delta_0}) \quad \mbox{as}\quad \tau\rightarrow \infty,\quad\varphi\in ]-\beta/2,\beta/2\,[.
\en

{\bf Step 2. Reduction to Laplace transforms.}
Assume that $v_1 \not\equiv 0$. Since $v_1$ is a solution of the Helmholtz equation in $B_{R}$, the lowest order nontrivial homogeneous polynomial $H(x)$ in the Taylor expansion of $v_1$ around the origin is a harmonic function (see \cite[Lemma 2.4]{BLS}). Without loss of generality, we assume $H$ is of order $n$ for some $n\geq 0$, i.e.,
 \be\label{eq:8}
 v_1(x)=H(x)+K(x),\qquad K(x)=\mathcal{O}(|x|^{n+1})\quad\mbox{as}\quad |x|\rightarrow 0.
 \en
 Define $F$ to be the Laplace transform of $H$ in $W$,
\be\label{F}
F(z):=\int_{W} \exp(-z\cdot x)\,H(x)\,dx,
\quad
\en for $z\in\mathbb C^2$ such that $\real (z)\cdot\left(1,0\right)>\cos(\beta/2)$.
 Taking $z=\rho=\rho_{\tau,\varphi,\pm}$ and splitting $F(\rho)$ into two terms, we see
\be\nonumber
F(\rho)&=&\int_{S_{R/2}} \exp(-\rho\cdot x)\,H(x)\,dx+\int_{W\backslash S_{R/2}} \exp(-\rho\cdot x)\,H(x)\,dx\\ \label{eq:6}
&=&\int_{S_{R/2}} \exp(-\rho\cdot x)\,H(x)\,dx+\mathcal{O}(e^{-\tau\delta_1})
\en as $\tau\rightarrow \infty$ for some $\delta_1=\delta_1(R,\varphi_0)>0$.
By the assumption $q(O)\neq 1$, we may set $\eta:=q(O)-1\neq 0$. Inserting (\ref{eq:7}) and (\ref{eq:8}) into
(\ref{eq:3}) and then combining the resulting expression with
 (\ref{eq:6}) gives
\ben
\eta\,F(\rho)=\int_{S_{R/2}} \exp(-\rho\cdot x)\,\bigl(\eta\,H(x)-(q(x)-1)(H(x)+K(x))(1+\psi(x))\bigr)\,dx+\mathcal{O}(e^{-\tau\delta_2})
\enn
as $\tau\rightarrow\infty$, with $\delta_2:=\min\{\delta_0,\delta_1\}$. Making use of \cite[Lemma 3.6]{BLS}, we can estimate the integral on the right hand of the previous identity by (see e.g., \cite[Section 4]{PSV})
\ben
&&\int_{S_{R/2}} \exp(-\rho\cdot x)\,\bigl(\eta\,H(x)-(q(x)-1)(H(x)+K(x))(1+\psi(x))\bigr)\,dx\\
&=& \int_{S_{R/2}} \exp(-\rho\cdot x)\,\left\{(q(O)-q(x)) H(x)-(q(x)-1)[K(x)+\psi(x)(H(x)+K(x))]\right\} \,dx\\
&\leq&C\,\tau^{-n-2-\delta}
\enn for some $\delta>0$. Therefore, we arrive at
\be\label{eq:9}
\eta\,F(\rho)\leq C\,\tau^{-n-2-\delta}
\en
for all $\varphi\in]-\beta/2,\beta/2\,[$ and $\tau>0$ sufficiently large.

On the other hand, since the cone $W$ remains invariant under the transform $x\rightarrow |\rho| x$ and $H$ is a homogeneous polynomial, it is easy to check that
\be\label{eq:4}
F(\rho)=|\rho|^{-n-2} \;F(\rho/|\rho|).
\en
Consequently, taking $\tau\rightarrow \infty$ in (\ref{eq:9}) gives $F((\omega+i\omega^\perp_\pm)/\sqrt{2})=0$. Moreover, the homogeneity of $F$ shown as in (\ref{eq:4}) yields
\be\label{eq:5}
F(\tau(\omega+i\omega^\perp_\pm))=0 \quad\mbox{for all}\quad \tau>0,\quad \varphi\in\;]-\beta/2,\beta/2\,[.
\en This implies the vanishing of the Laplace transform of $\chi_W H$ at $z=\tau(\omega+i\omega^\perp_\pm)$  for all $\tau>0$ and $\varphi\in\;]-\beta/2,\beta/2\,[$.

{\bf Step 3. End of the proof.}
Repeating the arguments of \cite[Section 5]{PSV}, one can deduce from (\ref{eq:5}), taking both signs $\pm$, that $H\equiv0$. This implies that $v_1\equiv 0$ in $B_R$. As a consequence, the Cauchy data of $v_2$ on $\Gamma^\pm_R$ vanish due to
the transmission conditions (\ref{Trans}). Finally, we get $v_2\equiv0$ by the unique continuation of solutions to the Schr\"odinger equation. This finishes the proof of Lemma \ref{lem2}.
\end{proof}

We are now ready to prove Theorems \ref{shape-identification-in-2D} and \ref{corner-scattering-in-2D} for general incident waves, including point source waves.
\begin{proof}[Proof of Theorem \ref{shape-identification-in-2D}.]  Since
$u_1^\infty (\hat{x})=u_2^\infty(\hat{x})$ for all $ \hat{x}\in \s^{1}$,
applying Rellich's lemma we know that $u_1^{\mathrm{sc}} = u_2^{\mathrm{sc}}$ in $\R^2\backslash(\overline{D}_1\cup\overline{D}_2)$. Thus
\be\label{eq:19}
u_1 (x)=u_2(x)
 \en
 for all $x\in \R^2\backslash(\overline{D}_1\cup\overline{D}_2)$.
\begin{figure}[htbp]
\centering
\begin{center}
\scalebox{0.35}{\includegraphics{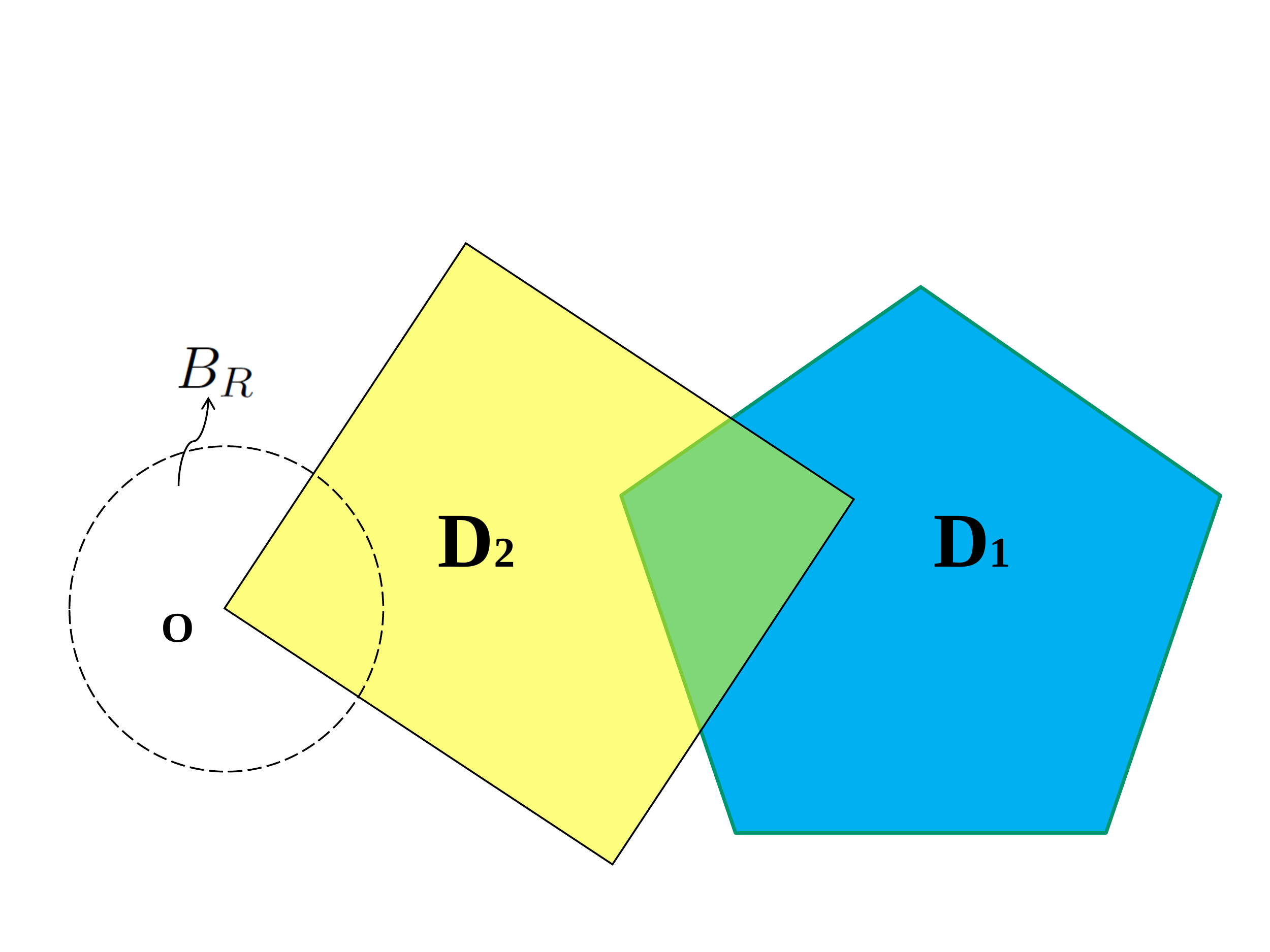}}
 \caption{Two distinct convex penetrable scatterers $D_1$ and $D_2$ of polygonal-type.}
\label{fig-3}
\end{center}
\end{figure}

If $\partial D_1\neq \partial D_2$, without loss of generality we may assume there exists a corner $O\in \R^2$ of $\partial D_2$ such that $O\notin \overline{D}_1$. We suppose further that  this corner point coincides with the origin and we pick
a fixed number $R>0$ such that $B_R\subset D_1^e$. Since $D_2$ is a convex polygon,
rotating coordinate axes if necessary, we may assume that $D_2\cap B_R=\{(r,\varphi): |\varphi|<\varphi_0\}$ for some $\varphi_0\in(0,\pi/2)$; see Figure \ref{fig-3}.
From (\ref{eq:19}), it follows that
\ben
u_1^-=u_1^+=u_2^+=u_2^-,\quad \partial_\nu u_1^-=\partial_\nu u_1^+=\partial_\nu u_2^+=\partial_\nu u_2^-\qquad\mbox{on}\quad \partial D_2\cap B_R,
\enn
where the superscripts $(\cdot)^-$, $(\cdot)^+$  stand for the limits taken from $D_2$ and $D_2^e$, respectively.
On the other hand,
 the function $u_1$ satisfies the Helmholtz equation with the wave number $k^2$ in $B_R$, while $u_2$ fulfills the Schr\"odinger equation
 \ben
 \Delta u_2+k^2 q_2 u_2=0\quad\mbox{in}\quad B_R.
 \enn
 Since $q_2(O)\neq 1$, applying Lemma \ref{lem2} leads to $u_1=u_2\equiv 0$ in $B_R$. Moreover, by unique continuation we obtain $u_1=u_2\equiv 0$ in $\R^2$. This implies that the scattered fields satisfy $u_1^{\mathrm{sc}} = u_2^{\mathrm{sc}} = -u^{\mathrm{in}}$ in all of $\R^2$. Hence $u^{\mathrm{in}}\equiv0$ in $\R^2\backslash(\overline{D}_1\cup\overline{D}_2)$, but since $u^{\mathrm{in}}$ solves the free Helmholtz equation, unique continuation implies $u^{\mathrm{in}} \equiv 0$. This contradiction gives $D_1=D_2$.
\end{proof}

\begin{proof}[Proof of Theorem \ref{corner-scattering-in-2D}.] Let us consider the incident wave $u^{\mathrm{in}}\not\equiv0$ with the total wave $u$, so that we have
\[\Delta u^{\mathrm{in}}(x)+k^2u^{\mathrm{in}}(x)=0\quad\text{and}\quad
\Delta u(x)+k^2q(x)u(x)=0\quad\text{in $\mathbb R^2$.}\]
If for this incident wave $u^\infty\equiv0$, then Rellich's lemma tells us that $u\equiv u^{\mathrm{in}}$ in $\mathbb R^2\setminus W$, and Lemma \ref{lem2} applied with $v_1 = u^{\mathrm{in}}$ and $v_2 = u$ gives $u^{\mathrm{in}} \equiv 0$ in $B_R$. By unique continuation we get $u^{\mathrm{in}} \equiv 0$, which is a contradiction.
\end{proof}

\section{Proofs in higher dimensions}\label{Proof-3D}

We first present the proof under the assumption (a) of Theorem \ref{shape-identification-in-3D}, that is, $D_j \subset \R^3$ is a rectangular box and the potential is $C^{\alpha}$ near the corners with $\alpha>1/4$. Let $W=\left[0,\infty\right[^3$. We will make use of the following result concerning complex geometrical optics solutions.
Recall the notation $\left\langle x\right\rangle:=(1+|x|^2)^{1/2}$.
 \begin{lem}\label{lem3.1}
 Let $\widetilde q\equiv1$ in $\mathbb R^3\setminus W$ and $\left\langle\cdot\right\rangle^\beta\left(\widetilde q-1\right)\in C^\alpha( W)$ for some $\alpha>1/4$ and $\beta>9/4$. If $\rho\in\mathbb C^3$ satisfies $\rho\cdot\rho=-k^2$ and $\left|\ima(\rho)\right|$ is sufficiently large, then there exists a solution of the Helmholtz equation
 \[\Delta u(x)+k^2\widetilde q(x)u(x)=0\quad\text{in}\quad\mathbb R^3\]
 of the form
 \[u=e^{-\rho\cdot x}(1+\psi(x)),\]
 where $\psi$ satisfies
 \[\left\|\psi\right\|_{L^4(\mathbb R^3)}=\mathcal O(\left|\ima(\rho)\right|^{-3/4-\delta}),\quad\text{as}\quad\left|\rho\right|\rightarrow\infty,\]
for some $\delta>0$.
\end{lem}
\noindent The proof of the above lemma is also based on the uniform Sobolev estimates of Kenig, Ruiz and Sogge \cite{KRS}.
In order to avoid repeating the arguments presented in \cite{PSV}, we shall verify Lemma  \ref{lem3.1} by
indicating the changes necessary to the proof of \cite[Theorem 3.1]{PSV}.
For this purpose we need to know into which Sobolev spaces the characteristic function of a cube belongs.
Below we will write $\chi_Q$ for the characteristic function of a set $Q$, and denote by $C_{\mathrm c}^\infty(\R^3)$ the space of smooth functions with compact support.

\begin{lem}\label{cube-lemma}
Let $Q\subset\mathbb R^3$ be a closed cube. Then
\[\chi_Q\in H^{\tau,p}(\mathbb R^3)\]
for $\tau\in\left[0,1/2\right[$ and $p\in\left]1,2\right]$.
\end{lem}

\begin{proof}
 Without loss of generality, we consider $Q=\left[-1,1\right]^3$. Since $\chi_Q\in L^1(\mathbb R^3)$, the Fourier transform $\widehat{\chi_Q}$ of $\chi_Q$ is continuous. For $\xi=(\xi_1,\xi_2,\xi_3)\in\mathbb R^3$ with $\xi_1\xi_2\xi_3\neq0$, the transform $\widehat{\chi_Q}$ takes the explicit form
\[\widehat{\chi_C}(\xi)=\frac{2^3\sin\xi_1\sin\xi_2\sin\xi_3}{\xi_1\xi_2\xi_3}.\]
Thus, we may estimate
\[\left\|\chi_Q\right\|_{H^{\tau,2}(\mathbb R^3)}^2
=\int_{\mathbb R^3}\left\langle\xi\right\rangle^{2\tau}\left|\widehat{\chi_C}(\xi)\right|^2\mathrm d\xi
\leq C\,\int_{\mathbb R^3}\left\langle\xi\right\rangle^{2\tau}\left\langle\xi_1\right\rangle^{-2}\left\langle\xi_2\right\rangle^{-2}\left\langle\xi_3\right\rangle^{-2}\mathrm d\xi.\]
The last integral is finite when $\tau<1/2$. Thus, $\chi_Q\in H^{\tau,2}(\mathbb R^3)$ for $\tau\in\left[0,1/2\right[$.

Next, let $\psi$ be a fixed function in $C_{\mathrm c}^\infty(\mathbb R^3)$ satisfying $\psi\chi_Q=\chi_Q$. By H\"older's inequality, for any $k\in\left\{0,1\right\}$, $f\in H^{k,2}(\mathbb R^3)$ and fixed $p\in\left]1,2\right]$,
\[\left\|\psi f\right\|_{H^{k,p}(\mathbb R^3)}
\leq C\,\left\|\psi f\right\|_{H^{k,2}(\mathbb R^3)}
\leq C\,\left\|f\right\|_{H^{k,2}(\mathbb R^3)}.\]
By interpolation, the mapping $f\mapsto\psi f$ maps $H^{s,2}(\mathbb R^3)$ into $H^{s,p}(\mathbb R^3)$ for $s\in\left]0,1\right[$, and thus $\chi_Q\in H^{\tau,p}(\mathbb R^3)$
for all $\tau\in\left[0,1/2\right[$ and $p\in\left]1,2\right]$.
\end{proof}

The construction of CGO solutions for a cube is proved as follows.

\begin{proof}[Proof of Lemma \ref{lem3.1}]
We may assume that $\alpha<1/2$. The proof of the complex geometric optics construction in \cite{PSV} is mostly independent of the shape of $W$. For $W=\left[0,\infty\right[^3$ we only need to check that $V:=\chi_W(1-\tilde{q})$ has the pointwise Sobolev multiplier property of Proposition 3.4 in \cite{PSV}, i.e., we need to check that
\[\left\|Vf\right\|_{H^{\alpha-\varepsilon,4/3}(\mathbb R^3)}\leq C \left\|f\right\|_{H^{\alpha-\varepsilon,4}(\mathbb R^3)} \]
for some constant $C>0$ for arbitrarily small fixed $\varepsilon>0$.
The desired multiplier property in turn follows immediately, if we can show that
(cf. \cite[Proposition 3.7]{PSV})
\be\label{eq}
\left\langle\cdot\right\rangle^{-\gamma}\chi_{\left[0,\infty\right[^3}\in H^{\tau,p}(\mathbb R^3)
\en
for $p\in\left]1,2\right]$, $\tau\in\left[0,1/2\right[$ and $\gamma\in\left]3/p,\infty\right[$.
Given $\beta_1,\beta_2\in\left[0,\infty\right[$ with $\beta_1<\beta_2$, we set
$\Lambda=\left[0,\beta_2\right]^3\setminus\left[0,\beta_1\right[^3$.
Applying Lemma \ref{cube-lemma} we
know that the function $\chi_{\Lambda}$ belongs to $H^{\tau,p}(\mathbb R^3)$.
This leads to the relation (\ref{eq}) by changing variables and scaling the Sobolev norm; see
 the proof for Proposition 3.7 in \cite{PSV}.
\end{proof}

To continue the proof of Theorem \ref{shape-identification-in-3D} under the assumption (a),
we again introduce some notation. Let $S_R=W\cap B_R$, $S_R^e=B_R\setminus S_R$, and $\Gamma_R=\partial W\cap B_R$.
As in two dimensions, we will employ a result of the following type.

\begin{lem}\label{3D-lemma}
Let $q \in L^{\infty}(B_R)$ satisfy $q|_{\overline S_R} \in C^\alpha(\overline S_R)$, where $\alpha>1/4$, and  $q\equiv1$ in $S_R^e$. Let $v_1,v_2\in H^2(B_R)$ be solutions to
\[\Delta v_1(x)+k^2v_1(x)=0\quad\text{and}\quad
\Delta v_2(x)+k^2q(x)v_2(x)=0\quad\text{in $B_R$},\]
subject to the transmission conditions
\[v_1=v_2,\quad\partial_\nu v_1=\partial_\nu v_2\quad\text{on $\Gamma_R.$}\]
Then we have $v_1= v_2\equiv0$ in $B_R$, if $q(O)\neq1$.
\end{lem}

\begin{proof}
We carry over the proof of Lemma \ref{lem2} to three dimensions.
 Set $w=v_1-v_2$. Then we have
$w=\partial_\nu w=0$ on $\Gamma_R$ and
\[\Delta w(x)+k^2q(x)w(x)=k^2(q(x)-1)v_1(x)\quad\text{in $B_R$.}\]
Extending $q$ from $B_{R/2}$ to $\mathbb R^3$, we can obtain a new potential $\widetilde q\in C^\alpha(\overline W)$ such that
$\widetilde q=q$ in $\overline S_{R/2}$ and that $\widetilde q-1$ satisfies the smoothness conditions required by
Lemma \ref{lem3.1}.
Next, write $\beta=\pi/3$ and $a=(3^{-1/2},3^{-1/2},3^{-1/2})$. Choose
 $\tau\in\mathbb R_+$ and $\omega,\omega^\bot\in\mathbb R^3$ with $\left|\omega\right|=\left|\omega^\bot\right|=1$, $\omega\cdot\omega^\bot=0$ and $\omega\cdot a>\cos(\beta/2)$. We parameterize the CGO solutions with the complex vector
\[\rho=\rho_{\tau,\omega,\omega^\bot}=\tau\omega+i(\tau^2-k^2)^{1/2}\omega^\bot.\]
 Provided that $\tau$ is sufficiently large, Lemma \ref{lem3.1} gives solutions
\[u(x)=u_{\tau,\omega,\omega^\bot}(x)=e^{-\rho\cdot x}(1+\psi(x))\]
to the Helmholtz equation
$\Delta u(x)+k^2\widetilde q(x)u(x)=0$ in $\mathbb R^3$.
Furthermore, the remainder $\psi$ has the $L^4$-estimate $\bigl\|\psi\bigr\|_{L^4(\mathbb R^3)}<C\,\tau^{-3/4-\delta}$ for some $\delta>0$.
Arguing as before in the two-dimensional case, we get
\[0=k^2\int_{S_{R/2}}(q-1)v_1u\,dx+\int_{\Lambda_{R/2}}\left( (\partial_\nu u)\,w-u\,\partial_\nu w\right)ds,\]
where $\Lambda_{R/2}=W\cap\partial B_{R/2}$. The selection of the parameters $\beta, a$ and $\omega$ ensures the decay of the integral over $\Lambda_{R/2}$,
\ben
\int_{\Lambda_{R/2}}\left((\partial_\nu u)w-u\partial_\nu w\right)ds=O(e^{-\tau\delta_0})\quad\mbox{as}\quad
\tau\rightarrow\infty
\enn
for some $\delta_0>0$.
Thus, we again get the orthogonality relation
\[\int_{S_{R/2}}(q-1)v_1u\,dx=O(e^{-\tau\delta_0})\quad\mbox{as}\quad
\tau\rightarrow\infty\]
 for any given admissible $\omega$ and $\omega^\bot$.

Denote by $H$ the lowest degree nontrivial homogeneous polynomial in the Taylor expansion of $v_1$ around the origin, and consider the
the Laplace transform $F(z)$ of $H$ (see (\ref{F})) for $z\in \C^3$ such that
 $\real (z)\cdot a>\cos(\beta/2)$. If $H$ has degree $n$, similarly as before we obtain
$F(\rho)=O(\tau^{-n-3-\delta})$
as $\tau\rightarrow\infty$, if $q(O)\neq 1$. This estimate involves using the H\"older inequality so that the $L^4$-norm of $\psi$ appears. In the other direction, by homogeneity, we have
\[F(\rho)=\left|\rho\right|^{-n-3}F\!\left(\frac\rho{\left|\rho\right|}\right).\]
 Taking $\tau\rightarrow\infty$ gives then
$F(\tau\omega+i\tau\omega^{\perp})=0$
for all $\tau\in\mathbb R_+$ and all admissible $\omega$ and $\omega^\bot$.
From \cite[Theorem 2.5]{BLS} we  obtain the conclusion $H(x)\equiv0$, which implies $v_1= v_2\equiv0$ in $B_R$.
\end{proof}

\begin{proof}[Proof of Theorem \ref{shape-identification-in-3D} under assumption (a)]
The result follows by the same arguments as in the proof of Theorem \ref{shape-identification-in-2D}, except that Lemma \ref{3D-lemma} is used instead of Lemma \ref{lem2}.
\end{proof}

Next we indicate how to prove Theorem \ref{shape-identification-in-3D} under the assumption (b). Now we let $W=\left[0,\infty\right[^N$ where $N \geq 3$. The required complex geometrical optics solutions were constructed in \cite{BLS} and they are given by the following result.

 \begin{lem}\label{lem3.1_second}
Let $\widetilde q\equiv1$ in $\mathbb R^N\setminus (W \cap B_R)$ for some $R > 0$, and let $\widetilde q\,|_{W \cap B_R}$ be in $H^{s,p}$ where $1 < p \leq 2$ and $s > N/p$. Let also $D \subset \R^N$ be a bounded open set, and let $2 \leq r < \infty$. If $\rho\in\mathbb C^N$ satisfies $\rho\cdot\rho = 0$ and $\left|\ima(\rho)\right|$ is sufficiently large, then there exists a solution of the Helmholtz equation
 \[\Delta u(x)+k^2\widetilde q(x)u(x)=0\quad\text{in}\quad D\]
 of the form
 \[u=e^{-\rho\cdot x}(1+\psi(x)),\]
 where $\psi$ satisfies
 \[\left\|\psi\right\|_{L^r(D)}=\mathcal O(\left|\ima(\rho)\right|^{-1}),\quad\text{as}\quad\left|\rho\right|\rightarrow\infty.\]
\end{lem}
\begin{proof}
We can write $\widetilde{q} = 1 - \chi_K \varphi$ for some cube $K = [0,a]^N$ and for some $\varphi \in H^{s,p}_{\mathrm c}(\R^N)$ by the conditions on $\widetilde{q}$ and the Sobolev extension theorem on Lipschitz domains. Writing $m = \chi_K \varphi$, the equation that we need to solve is
\[
(\Delta + k^2(1-m)) u = 0 \qquad \text{in $D$}.
\]
The result would then follow from \cite[Theorem 2.3]{BLS}, except that this theorem was proved under the condition $\varphi \in C^{\infty}$ instead of $\varphi \in H^{s,p}$.

Inspecting the proof in \cite{BLS} we see that it is enough that the function
\[
Q = -k^2(1-m) \Phi_D,
\]
where $\Phi_D \in C^{\infty}_{\mathrm c}(\R^N)$ satisfies $\Phi_D = 1$ near $\overline{D}$, satisfies \cite[formula (34)]{BLS}, i.e.\ that one has
\begin{equation} \label{bps_formula34}
Q \in \widehat{B^1_{r,1}} \quad \text{and} \quad \norm{Qg}_{\widehat{B^1_{r,1}}} \leq C_Q \norm{g}_{\widehat{B^{-1}_{r,\infty}}}
\end{equation}
where the spaces are as in \cite{BLS}. Now we can write $Q = Q_1 + Q_2$ where $Q_1 = -k^2 \Phi_D \in C^{\infty}_{\mathrm c}(\R^N)$ and $Q_2 = f \chi_K$ where
\[
f = k^2 \varphi \Phi_D,
\]
so that $f \in H^{s,p}_{\mathrm c}(\R^N)$. We use \cite[Lemma 4.3 and preceding discussion]{BLS} to conclude that when $\mathrm{supp}(q) \subset B_R$, one has
\begin{align*}
\norm{q}_{\widehat{B^1_{r,1}}} &\leq C R \norm{\hat{q}}_{L^r}, \\
\norm{q g}_{\widehat{B^1_{r,1}}} &\leq 2R^2 \norm{\hat{q}}_{L^1} \norm{g}_{\widehat{B^{-1}_{r,\infty}}}, \\
\norm{\chi_K g}_{\widehat{B^1_{r,1}}} &\leq C_r \norm{g}_{\widehat{B^1_{r,1}}}.
\end{align*}

It follows that \eqref{bps_formula34} will be satisfied if the function $f$ defined above satisfies $\hat{f} \in L^1 \cap L^r$. Since $f \in L^1_{\mathrm c}(\R^N)$, we have $\hat{f} \in L^{\infty}(\R^N)$ and it is enough to check that $\hat{f} \in L^1$. But if $(\psi_j(\xi))_{j=0}^{\infty}$ is a Littlewood--Paley partition of unity and if $1 \leq p \leq 2$, we obtain by the H\"older and Hausdorff--Young inequalities that
\begin{align*}
\int_{\R^N} \abs{\hat{f}(\xi)} \,d\xi &= \sum_{j=0}^{\infty} \int_{\R^N} \psi_j(\xi) \abs{\hat{f}(\xi)} \,d\xi
\leq C \sum_{j=0}^{\infty} 2^{jN/p} \norm{\psi_j(\xi) \hat{f}(\xi)}_{L^{p'}} \\
&\leq C \sum_{j=0}^{\infty} 2^{jN/p} \norm{\psi_j(D) f}_{L^p}
\leq C \norm{f}_{B^{N/p}_{p,1}}
\end{align*}
where $\psi_j(D)$ is the Fourier multiplier with symbol $\psi_j(\xi)$ and the last norm is a Besov norm. Since $H^{s,p} \subset B^{s}_{p,\infty} \subset B^{N/p}_{p,1}$ for $s > N/p$, we get $\hat{f} \in L^1$ as required. This shows that \eqref{bps_formula34} is satisfied, which concludes the proof.
\end{proof}

\begin{proof}[Proof of Theorem \ref{shape-identification-in-3D} under assumption (b)]
It is enough to use Lemma \ref{lem3.1_second} to prove an analogue of Lemma \ref{3D-lemma}, and then argue as in the proof of Theorem \ref{shape-identification-in-2D}.
\end{proof}

\section{Concluding remarks}
In two dimensions, we have verified the uniqueness in identifying a convex penetrable scatterer of polygonal type with a single far-field pattern, provided the refractive index is discontinuous at the corner points but $C^\alpha$-H\"older continuous inside near the corners. In higher dimensions, the uniqueness applies to convex polyhedra with additional assumptions on the geometrical shape (i.e., boxes) and on the smoothness of the contrast.
In this study,
the smoothness assumption is required only near the corner points.

Our future efforts will be devoted to the uniqueness proof in 3D for convex polyhedra with general $C^\alpha$-H\"older ($\alpha>0$) continuous potentials. Since the CGO solutions can be constructed with plenty of generality, the 3D proof essentially requires to evaluate the Laplace transform of a harmonic homogeneous polynomial over a general three-dimensional corner domain and then to prove the vanishing of this polynomial through novel techniques. Another possible approach would be to analyze the corner and edge singularities of an elliptic equation with analytical Cauchy data in weighted H\"older spaces. Further results will be presented in a forthcoming publication. 

\section*{Acknowledgment}

The first author would like to acknowledge the support from the German Research Foundation (DFG) under Grant No.\ HU 2111/1-2. The second author was partly supported by an ERC Starting Grant (grant agreement no.\ 307023) and CNRS. The second and third authors were also supported by the Academy of Finland (Centre of Excellence in Inverse Problems Research), and they would like to thank the Institut Henri Poincar\'e Program on Inverse Problems in 2015 where part of this work was carried out.

\bibliographystyle{alpha}

\end{document}